\documentclass[english]{article}

\usepackage{babel}

\usepackage[a4paper,top=2cm,bottom=2cm,left=3cm,right=3cm,marginparwidth=1.75cm]{geometry}

\usepackage{tikz-cd} 
\usepackage{amssymb,amsfonts,amsmath,amsthm}
\usepackage{graphicx}
\usepackage[colorlinks=true, allcolors=blue]{hyperref}
\usepackage{mathrsfs}
\usepackage{enumitem}

\usepackage[cmtip,arrow]{xy}
\usepackage{pb-diagram,pb-xy}

\usepackage{newfloat}
\DeclareFloatingEnvironment[fileext=lod,name=Diagram,placement=p]{fdiagram}

\numberwithin{equation}{section}

\newtheorem{theorem}{Theorem}[section]

\newtheorem{corollary}[theorem]{Corollary}
\newtheorem{lemma}[theorem]{Lemma}
\newtheorem{proposition}[theorem]{Proposition}
\theoremstyle{definition}
\newtheorem{definition}[theorem]{Definition}
\theoremstyle{remark}
\newtheorem{remark}[theorem]{Remark}
\newtheorem*{remark*}{Remark}

\newcommand{\Z}{{\mathbb Z}}

\newcommand{\R}{{\mathbb R}}
\newcommand{\N}{{\mathbb N}}

\DeclareMathOperator{\Inn}{Inn}
\DeclareMathOperator{\Aut}{Aut}
\DeclareMathOperator{\Out}{Out}
\DeclareMathOperator{\End}{End}
\DeclareMathOperator{\GL}{GL}
\DeclareMathOperator{\GO}{O}
\DeclareMathOperator{\SO}{SO}
\DeclareMathOperator{\FR}{FR}
\DeclareMathOperator{\diag}{diag}

\title{Automorphism groups of combinatorial Hantzsche-Wendt groups}

\date{}

\usepackage[noblocks]{authblk}

\usepackage[hang,flushmargin]{footmisc}
\author{Rafał Lutowski$^{*,}$}
\author{Andrzej Szczepański}
\affil{Institute of Mathematics, Faculty of Mathematics, Physics and Informatics, University of Gda\'nsk\par Wita Stwosza 57, 80-308 Gda\'nsk, Poland}
\author{Richard Weidmann}
\affil{Mathematisches Seminar, Christian-Albrechts-Universit\"at zu Kiel\par Heinrich-Hecht-Platz 6, 24118 Kiel, Germany}

\begin{document}
\maketitle

\renewcommand{\thefootnote}{\fnsymbol{footnote}}
\begin{NoHyper}
\footnotetext{$^*$Corresponding author: \url{rafal.lutowski@ug.edu.pl}}
\end{NoHyper}
\renewcommand{\thefootnote}{\arabic{footnote}}

\begin{abstract}
\noindent
Combinatorial Hantzsche-Wendt groups were introduced by W. Craig and P.A. Linnell. Every such a group $G_n$, where $n$ is a natural number, encodes the holonomy action of any $n+1$-dimensional Hantzsche-Wendt manifold. $G_2$ is the fundamental group of the classical Hantzsche-Wendt manifold -- the only one $3$-dimensional oriented flat manifold with non-cyclic holonomy group. In this article, we describe the structure of the automorphism and of the outer automorphism groups of combinatorial Hantzsche-Wendt groups.
\end{abstract}

\section{Introduction}

A \emph{crystallographic group} of dimension $n$ is a discrete and cocompact subgroup of the group $E(n) = \GO(n)\ltimes \R^n$ of isometries of the Euclidean space $\R^n.$ By Bieberbach theorems (see \cite{B11,B12}, \cite{Ch86}, \cite{Sz12}), any crystallographic group $\Gamma$ of dimension $n$ defines a short exact sequence
\[
0 \longrightarrow \Z^n \longrightarrow \Gamma \longrightarrow H \longrightarrow 1,
\] 
where $\Z^n$ is isomorphic to the subgroup of all translations of $\Gamma$ and $H$ is a finite group.

By definition, $\Gamma$ acts on the Euclidean space $\R^n.$ When this action is free or -- equivalently -- when $\Gamma$ is torsion-free, the orbit space $M = \R^n/\Gamma$ is a connected closed Riemannian manifold with vanishing sectional curvature -- a \emph{flat manifold}. In this case, $\Gamma$ is called a \emph{Bieberbach group}  and $\Gamma$ is the fundamental group of $M$. $M$ is orientable if and only if $\Gamma \subset \SO(n) \ltimes \R^n$.

Assume that $M = \R^n/\Gamma$ is an orientable flat manifold, where $H = C_2^{n-1}$ is an elementary abelian group of rank $n-1$. Then $\Gamma$ and $M$ are called \emph{Hantzsche-Wendt} group and manifold, respectively. They exist only in odd dimensions (see \cite{MR99}). For the specific case $n = 3$, up to isomorphism, there exists only one Hantzsche-Wendt group -- $\Gamma_3$. It was defined for the first time by W. Hantzsche, H. Wendt and independently by W. Nowacki in 1934 (see \cite{HW31}, \cite{NW}). The presentation of $\Gamma_3$ follows from
\cite[Lemma 13.3.1, pp. 606-607]{Pass77}:
\[
\Gamma_3 = \langle x,y | x^{-1}y^2xy^2, y^{-1}x^2yx^2\rangle.
\]
The group $\Gamma_3$ has many interesting properties -- it is a non-unique product group (see \cite{Prom33}) and a counterexample to the Kaplansky unit conjecture (see \cite{gard21}, \cite{Gard23}). For more properties of Hantzsche-Wendt groups, we refer the reader to \cite[Chapter 9]{Sz12}.

In \cite{CL22}, the following generalization of $\Gamma_3$ is proposed:
\begin{definition}
\label{maindef}
Let $n \in \N$. The \emph{combinatorial Hantzsche-Wendt group} $G_n$ is given by the following presentation:
\[
G_n := \langle x_1, \ldots, x_n \mid \forall_{i \neq j} \; x_i^{-1} x_j^2 x_i = x_j^{-2}  \rangle.
\]
\end{definition}
By \cite[Lemma 3.1]{CL22}, $A_n := \langle x_1^2, \ldots, x_n^2 \rangle$ is a normal free abelian subgroup of rank $n$ and we have the following short exact sequence
\begin{equation}
\label{eq:extension}
1 \longrightarrow A_n \longrightarrow G_n \longrightarrow W_n=G_n/A_n \longrightarrow 1.
\end{equation}
If \ $\bar{\ } \colon G_n \to W_n$, $g\mapsto \bar g=gA_n$ denotes the canonical epimorphism, then
\[
W_n=\langle \bar x_1,\ldots ,\bar x_n\mid \bar x_1^2,\ldots ,\bar x_n^2\rangle\cong \underbrace{C_2* \ldots * C_2}_n
\]
is the free product of $n$ copies of the group of order $2$.

The goal of the article is a description of the automorphism and outer automorphism groups of combinatorial Hantzsche-Wendt groups, by revealing their structure in Theorems \ref{theorem:aut-out-1} and \ref{theorem:aut-out-2} as well as their presentations in Theorems \ref{theorem:presentation-of-autg} and \ref{theorem:presentation-of-outg}.

To simplify the notation, we fix the positive integer $n$ throughout the article, and we assume the following notation:
\[
A := A_n, G := G_n, W := W_n.
\]

Let us present the structure of our article. Section 2 is devoted to a presentation of the group $\Aut(W)$. Section 3 deals with specific endomorphisms of the group $G$. Sections 4 and 5 give us the first view of the structure of the automorphism group of $G$. In particular, we use a Lemma \ref{lemma:characteristic} to prove that $A$ is a characteristic subgroup of $G.$ To take advantage of Charlap's approach in the determining of $\Out(G)$ from \cite{Ch86}, presented in Section 7, we first deal with semi-linear automorphisms of $A$ in Section 6.

\section{Presentation of \texorpdfstring{$\Aut(W)$}{Aut(W)}}
\label{section:presentation-of-autw}

The purpose of this section is to give a presentation of the group $\Aut(W)$. A presentation of the automorphism group of an arbitrary free product in terms of the automorphism groups of the free factors was given by Fouxe-Rabinovitch \cite{Foux40,Foux41}, see also Gilbert \cite{Gi87}. 

In the current setting, the presentation will be much simpler than in the general case as $W$ has no cyclic free factors and as all free factors in the Grushko decomposition are of order $2$ and therefore have trivial automorphism group. Thus there are no non-trivial factor automorphisms. We extract a presentation for $\hbox{Aut}(W)$ from the one given in \cite{CG90} and \cite{An21} for the case that there are no infinite cyclic-free factors.

\smallskip
There are two types of generators for $\hbox{Aut}(W)$:

\begin{enumerate}
    \item For $\sigma\in S_n$ the automorphism $\bar\alpha_{\sigma}$ given be $\bar\alpha_\sigma(\bar x_i)=\bar x_{\sigma(i)}$. We refer to these automorphisms as permutation automorphisms.
    \item For $1\le i\neq j\le n$ the automorphism $\bar\alpha_i^j$ given by $\bar\alpha_i^j(\bar x_i)=\bar x_j\bar x_i\bar x_j^{-1}$ and $\bar\alpha_i^j(\bar x_k)=\bar x_k$ for $k\neq i$. Collins and Gilbert \cite{CG90} call these generators (among others) Fouxe-Rabinovitch generators. 
\end{enumerate}

The generators of the first type generate a subgroup isomorphic to $S_n$ which we again denote by $S_n$ and the Fouxe-Rabinovitch generators generate the Fouxe-Rabinovitch subgroup denoted by $\FR(W)$. Note that $\Aut(W)$ is generated by the $\bar\alpha_\sigma$ and $\bar\alpha_1^2$ as each $\bar\alpha_i^j$ is conjugate to $\bar\alpha_1^2$ by some $\bar\alpha_\sigma$. Using the fact that the above automorphisms do indeed give a generating set of $\Aut(W)$ we obtain a split short exact sequence 
\begin{equation}
\label{equation:aut_w}
1\to \FR(W)\to \hbox{Aut}(W)\to S_n\to 1.
\end{equation}

The following relators are sufficient to give a presentation of $\Aut(W)$ with the above generators. Note that (1) suffices to give a presentation of $S_n$, that (2)-(4) give a presentation of $\FR(W)$ by Proposition 3.1 of \cite{CG90} and that (5) describes the action of $S_n$ on $\FR(W)$.
\begin{enumerate}[label=(\arabic*)]
    \item $\bar\alpha_{\sigma} \circ \bar\alpha_{\tau} = \bar\alpha_{\sigma\tau}$ for all $\sigma,\tau\in S_n$ 
    \item $(\bar\alpha_i^j)^2 = \hbox{id}$  for all $1\le i\neq j\le n$.
    \item $\bar\alpha_i^j \circ \bar\alpha_k^l = \bar\alpha_k^l \circ \bar\alpha_i^j$ for all $1\le i,j,k,l\le n$ with $i\neq k$, $j\neq k$ and $l\neq i$  
    \item  $(\bar\alpha_i^j \circ \bar\alpha_m^j) \circ \bar\alpha_i^m = \bar\alpha_i^m \circ (\bar\alpha_i^j \circ \bar\alpha_m^j)$ for all $1\le i,j,m\le n$ with  $i,j,m$ pairwise distinct
    \item $\bar\alpha_{\sigma} \circ \bar\alpha_i^j = \bar\alpha_{\sigma(i)}^{\sigma(j)} \circ \bar\alpha_{\sigma}$ for all $1\le i\neq j\le n$ and $\sigma\in S_n$
\end{enumerate}

\section{Translation endomorphisms of \texorpdfstring{$G$}{G}}
\label{section:translation-endomorphisms}

In this section we study the structure of the group of  automorphisms of $G$ which arise from multiplying each generator by an element of $A$. To achieve it, we first deal with the monoid of such endomorphisms.

\begin{lemma}\label{newlemma} Let $a=(x_1^2)^{{z_1}} \cdots (x_n^2)^{{z_n}}\in A$. Then $(x_ia)^2=(x_i^2)^{2{z_i}+1}$
\end{lemma}

\begin{proof} Note first that the two occurrences of $(x_j^2)^{z_j}$ with $j\neq i$ cancel in  
\[
(x_ia)^2 = x_i^2 x_i^{-1} \cdot (x_1^2)^{{z_1}} \cdots (x_n^2)^{{z_n}} \cdot x_i\cdot  (x_1^2)^{{z_1}} \cdots (x_n^2)^{{z_n}},
\] 
as the sign of the exponent of the second occurrence changes if it is pulled past $x_i$. Thus 
\[
(x_ia)^2=x_i^2 x_i^{-1}(x_i^2)^{z_i}x_i(x_i^2)^{z_i}=(x_i^2)^{2{z_i}+1},
\] 
which proves the claim.
\end{proof}

\begin{lemma}
\label{lemma:shift-endomorphisms}
Let $\mathfrak a=[a_{ij}]\in M_n(\mathbb Z)$ be a square integer matrix of degree $n$. For $1\le i\le n$ let $a_i := (x_1^2)^{a_{i1}} \cdots (x_n^2)^{a_{in}}\in A$. Then there exists a unique endomorphism $t_{\mathfrak a}$ of $G$ such that
\[
t_{\mathfrak a}(x_i) = x_i a_i
\]
for all $1 \leq i \leq n$.
\end{lemma}

\begin{proof}
The uniqueness part is trivial as the $x_i$ generate $G$. Set $y_i := x_i a_i$ for $1\le i\le n$. To show that there exists an endomorphism $t_{\mathfrak a}:G\to G$ such that $t_a(x_i)=x_ia_i=y_i$ for $1\le i\le n$ it suffices to verify that $y_i^{-1} y_j^2 y_i = y_j^{-2}$ for all $i\neq j\in\{1,\ldots ,n\}$ by von Dyck's theorem. 

Let $i \neq j\in\{1,\ldots ,n\}$. It follows  from Lemma~\ref{newlemma} that
\begin{equation}
\label{equation:square-of-shift}
y_j^2 = (x_ja_j)^2 = (x_j^2)^{2a_{jj}+1} 
\end{equation}
and therefore
\[
y_i^{-1} y_j^2 y_i = a_i^{-1}x_i^{-1}(x_j^2)^{2a_{jj}+1}x_ia_i=a_i^{-1}(x_j^2)^{-(2a_{jj}+1)}a_i=a_i^{-1}y_j^{-2}a_i= y_j^{-2},
\]
which proves the claim.
\end{proof}

We call the map $t_{\mathfrak a}$ from Lemma \ref{lemma:shift-endomorphisms} a \emph{translation} or \emph{shift} endomorphism of $G$.

\medskip Let $M = M_n(\Z)$. We define a binary operation 
\[
* \colon M \times M \to M, (\mathfrak a,\mathfrak b)\mapsto \mathfrak a*\mathfrak b
\] 
on $M$ by setting $[a_{ij}] * [b_{ij}]$ to be the matrix $[c_{ij}]$ where 
\begin{equation}
\label{equation:monoid-action2}
c_{ij} = a_{ij}+(1+2a_{jj})b_{ij}
\end{equation} for $1\le i,j\le n$.

\begin{lemma}
\label{lemma:representation-of-monoid}
Let $\mathfrak a,\mathfrak b\in M$. Then $t_{\mathfrak a}\circ t_{\mathfrak b}=t_{\mathfrak a*\mathfrak b}$.    
\end{lemma}

\begin{proof}Let $\mathfrak a=[a_{ij}],\mathfrak b=[b_{ij}]\in M$ and $\mathfrak c=[c_{ij}]:=\mathfrak a*\mathfrak b$. We need to show that $t_{\mathfrak a}\circ t_{\mathfrak b}=t_{\mathfrak c}$. Using the definition of $t_{\mathfrak a}$ and $t_{\mathfrak c}$ and Lemma~\ref{newlemma} we obtain 
\begin{align*}
t_{\mathfrak a}\circ t_{\mathfrak b}(x_i) &=t_{\mathfrak a}(t_{\mathfrak b}(x_i))=t_{\mathfrak a}(x_ib_i)\\
&=t_{\mathfrak a}(x_i(x_1^2)^{b_{i1}}(x_1^2)^{b_{i2}}\cdots (x_n^2)^{b_{in}})\\
&=t_{\mathfrak a}(x_i)t_{\mathfrak a}(x_1^2)^{b_{i1}}\cdots t_{\mathfrak a}(x_n^2)^{b_{in}}\\
&=x_ia_i((x_1^2)^{2a_{11}+1})^{b_{i1}}\cdots ((x_n^2)^{2a_{nn}+1})^{b_{in}}\\
&=x_i(x_1^2)^{a_{i1}}\cdots (x_n^2)^{a_{in}}(x_1^2)^{(2a_{11}+1)b_{i1}}\cdots (x_n^2)^{(2a_{nn}+1){b_{in}}}\\
&= x_i(x_1^2)^{(2a_{11}+1)b_{i1}+a_{i1}}\cdots (x_n^2)^{(2a_{nn}+1)b_{in}+a_{in}}=t_{\mathfrak c}(x_i)
\end{align*}
for $1\le i\le n$,  which proves that $t_{\mathfrak a}\circ t_{\mathfrak b}=t_{\mathfrak c}$.
\end{proof}

\begin{corollary}
$(M,*)$ is monoid with identity element the zero matrix. Moreover the map 
\[
t \colon M\to \hbox{End}(G),\ \mathfrak a\mapsto t_{\mathfrak a}
\] 
is a monoid monomorphism.
\end{corollary}

Note that $t(M)$ is a submonoid of $\End(G)$ consisting of the translation endomorphisms of $G$.

\begin{remark}
\label{remark:decomposition-of-the-monoid}
Let $\mathfrak a= [a_{ij}] \in M$. Let
\[
\mathfrak a_0 = \begin{bmatrix}
     0 & a_{21} & \ldots & a_{1n}\\
a_{21} &      0 & \ldots & a_{2n}\\
\vdots & \vdots & \ddots & \vdots\\
a_{n1} & a_{n2} & \ldots &      0
\end{bmatrix}
\text{ and }
\mathfrak a_d = 
\begin{bmatrix}
a_{11} &      0 & \ldots &      0\\
     0 & a_{22} & \ldots &      0\\
\vdots & \vdots & \ddots & \vdots\\
     0 &      0 & \ldots & a_{nn}
\end{bmatrix}.
\]
By definition of the action in $M$ we get
\[
\mathfrak a = \mathfrak a_0 * \mathfrak a_d.
\]
Hence we have $M = M_0 * M_d$, where $M_0 := \{ \mathfrak a_0 : \mathfrak a \in M \}$ and $M_d := \{ \mathfrak a_d : \mathfrak a \in M \}$.
\end{remark}

\begin{lemma}
\label{lemma:invertitbles-in-monoid}
The set underlying the group of units of $M$ is  
\[
M^* = \left\{ [a_{ij}] \in M : \forall_{1 \leq i \leq n} \; a_{ii} \in \{0,-1\} \right\}.
\]
Moreover $t(M^*)$ is the group of all shift automorphisms of $G$.
\end{lemma}

\begin{proof}
Assume that $\mathfrak b = [b_{ij}]$ is a right inverse of $\mathfrak a = [a_{ij}]$. Let $1 \leq j \leq n$. Formula \eqref{equation:monoid-action2} shows in particular that
\[
b_{jj} = -\frac{a_{jj}}{1+2a_{jj}},
\]
which is an integer if and only if $a_{jj} \in \{0,-1\}$. We get that
\[
b_{ij} = \left\{
\begin{array}{rl}
-a_{ij} & \text{ if } a_{jj} = 0,\\
 a_{ij} & \text{ if } a_{jj} = -1.
\end{array}
\right.
\]
Obviously, $\mathfrak b$ is also a left inverse for $\mathfrak a$.

Now, let $t_{\mathfrak a}$ be a shift automorphism, for some $\mathfrak a \in M$. Write $\mathfrak a=\mathfrak a_0 *\mathfrak a_d$ as in Remark \ref{remark:decomposition-of-the-monoid}. By the above argument $\mathfrak a_0\in M^*$, thus $t_{\mathfrak a_0}$ is an automorphism.  Lemma \ref{lemma:representation-of-monoid} implies that $t_{\mathfrak a}=t_{\mathfrak a_0}t_{\mathfrak a_d}$, thus $t_{\mathfrak a_d} = (t_{\mathfrak a_0})^{-1}t_{\mathfrak a}$ is an automorphism of $G$. This can only happen if $\mathfrak a_d \in M^*$. Thus $\mathfrak a =\mathfrak a_0 *\mathfrak a_d\in M^*$ as desired.
\end{proof}

Let $M^*_d := M^* \cap M_d$. Using again the formula \eqref{equation:monoid-action2} one gets that the action $*$ restricted to $M_0$ is simply addition, hence 
\[
M_0 \cong\langle \{\varepsilon_{ij}:1\le i\neq j\le n\}\mid \{[\varepsilon_{ij},\varepsilon_{kl}]:1\le i\neq j,k\neq l\le n\}\rangle\cong \Z^{n(n-1)},
\] 
where $\varepsilon_{ij}$ is the matrix with entry $1$ in the $i$-th row and $j$-th column and zero elsewhere. Moreover 
\[
M^*_d \cong\langle \delta_1,\ldots,\delta_n\mid \delta_i^2, [\delta_i,\delta_j]\hbox{ for }1\le i,j\le n\rangle \cong C_2^n,
\] 
where $\delta_i$ is the matrix that has only zero entries except for a $-1$ entry in the $i$-th row and $i$-th column.

\smallskip
By Remark \ref{remark:decomposition-of-the-monoid} and Lemma \ref{lemma:invertitbles-in-monoid} we have $M^* = M_0 M^*_d$. Note that the conjugation of an element of $M_0$ by an element of $M^*_d$ may result only in a sign change in some columns. 
To be more precise, if $\mathfrak a = [ a^{(1)} \ldots a^{(n)} ] \in M_0$, where $a^{(i)}$ is the $i$-th column of $\mathfrak a$, then
\[
\delta_k * \mathfrak a * \delta_k^{-1} = 
\begin{bmatrix}
a^{(1)} & \ldots & a^{(k-1)} & -a^{(k)} & a^{(k+1)} & \ldots & a^{(n)}
\end{bmatrix}.
\]
Hence, we get
\begin{proposition}
\label{proposition:shift-automorphisms}
$M^*$ is a semidirect product of $M_0$ and $M^*_d$. A (relative) presentation of $M^*$ is given by 
\[
\langle M_0,M_d*\mid \{\delta_i\varepsilon_{ji}\delta_i=-\varepsilon_{ji}: 1\le i,j\le n\}\cup\{\delta_i\varepsilon_{jk}\delta_i=\varepsilon_{jk}: 1\le i,j,k\le n, k\neq i\}\rangle.
\]
\end{proposition}

\section{Automorphisms of \texorpdfstring{$G$}{G}}

We begin with a simple lemma that provides sufficient conditions for a subgroup to be characteristic. We denote the centralizer of an element $g\in G$ by $C_G(g)$.

\begin{lemma}
\label{lemma:characteristic}
Let $H$ be a group containing a normal subgroup $N$. Assume that:
\begin{enumerate}[label=(\arabic*)]
    \item \label{item:sol1} $N$ is solvable;
    \item \label{item:sol2} for every non-trivial element $\bar h \in H/N$, $C_{H/N}(\bar h)$ is solvable;
    \item \label{item:nsol} for every element $a \in N$, $C_{H}(a)$ is non-solvable.
\end{enumerate}
Then $N$ is a characteristic subgroup of $H$.
\end{lemma}
\begin{proof}
Suppose there exists $\alpha \in \Aut(H)$ such that $\alpha(N)\not\subset N$. Let $a \in \alpha(N) \setminus N$. Clearly \ref{item:nsol} implies that $C_H(a)$ is non-solvable as $a$ is the image of some element of $N$ under an automorphism of $H$.

We obtain the short exact sequence
\[
1 \longrightarrow N \cap C_H(a) \longrightarrow C_H(a) \longrightarrow N C_H(a) / N \longrightarrow 1.
\]
By our assumption, $a N\in H/N$ is non-trivial. Moreover $N C_H(a) / N$ is a subgroup of $C_{H/N}(aN)$ and therefore solvable by \ref{item:sol2}.
By \ref{item:sol1} $N\cap C_H(a)$ is solvable. Thus $C_H(a)$ is an extension of solvable groups and therefore solvable, a contradiction.

We have shown that $\alpha(N)\subset N$ for all $\alpha\in\hbox{Aut}(H)$. Thus $N$ is characteristic.
\end{proof}

\begin{corollary}
\label{corollary:characteristic}
Assume $n \geq 3$. Then $A$ is a characteristic subgroup of $G$.
\end{corollary}
\begin{proof}
We apply Lemma \ref{lemma:characteristic} with $H=G$ and $N=A$. Clearly \ref{item:sol1} is satisfied as $A$ is abelian. \ref{item:sol2} holds as the centralizer of any element in a free product of cyclic groups is cyclic. To verify \ref{item:nsol} it is enough to notice, that by \cite[Proposition 6]{PS22} the kernel of the action of $W$ on $A$, defined by 
\begin{equation}
\label{equation:action-of-w}
\forall_{w \in G} \forall_{a \in A} \bar{w} \cdot a = waw^{-1},
\end{equation}
is a non-abelian free group. Hence, $G$ itself contains a non-abelian free group which centralizes any element of $A$.
\end{proof}

\begin{remark}
The outer automorphism group of $G_2$ was computed in \cite{Zi90}. Although the author is not directly interested in calculating $\Aut(G_2)$, some of his intermediate results can be used to compute it, using for example tools developed in \cite{Lu13}. 
\end{remark}

Using the notation of Section \ref{section:translation-endomorphisms} we state the following theorem:

\begin{theorem}
\label{theorem:aut-w-1}
Let $n \geq 3$. We have the following short exact sequence:
\[
1 \longrightarrow M^* \stackrel{t}{\longrightarrow} \Aut(G) \stackrel{\pi}{\longrightarrow} \Aut(W) \longrightarrow 1.
\]
\end{theorem}
\begin{proof}
By the description given in Section \ref{section:translation-endomorphisms}, the restriction of $t$  to $M^*$ is a monomorphism. Since any element of the kernel of $\pi$ must be a translation automorphism, it follows that $t(M^*) = \ker \pi$.

We are left to show that every automorphism of $W$ can be lifted to an automorphism of $G$, it clearly suffices to show that a generating set of $\Aut(W)$ can be lifted. By the description in Section \ref{section:presentation-of-autw}, $\Aut(W)$ is generated by $\bar\alpha_\sigma$ for $\sigma \in S_n$ and $\bar\alpha_i^j$ for $i\neq j$. The automorphisms $\bar\alpha_\sigma\in\Aut(W)$ have obvious lifts to $\Aut(G)$ that we denote by $\alpha_\sigma$. 
It remains to check that $\bar\alpha_i^j$ can be lifted. Let $\eta\in\{-1,1\}$ and set
\[
y_i := x_j^\eta x_i x_j^{-\eta} \text{ and } y_k := x_k
\]
for $k\neq i$. It clearly suffices to check that the map 
\[
\alpha_\eta:\{x_1,\ldots ,x_n\}\to G,\ x_i\mapsto y_i
\] 
induces an automorphism $\alpha_\eta$ of $G$. We first check that in induces an endomophism, i.e. that $y_i$ satisfy the relations of $G$. For those relations not involving $x_i$ this is trivial. We check those involving $x_i$. Note that $y_i^2 = x_i^{-2}$. For $k\neq i$ we get
\[
y_i^{-1} y_k^2 y_i = x_j^{\eta} x_i^{-1} x_j^{-\eta} x_k^2 x_j^\eta x_i x_j^{-\eta} = x_k^{-2} = y_k^{-2}
\]
and
\[
y_k^{-1} y_i^2 y_k = x_k^{-1} x_i^{-2} x_k = x_i^2 = y_i^{-2},
\]
hence $\alpha_\eta$ induces an endomorphism. As $\alpha_\eta$ is clearly inverse to $\alpha_{-\eta}$ it follows that $\alpha_\eta$ is bijective which proves the claim. We will denote the lift $\alpha_1$ of $\bar\alpha_i^j$ by $\alpha_i^j$.
\end{proof}

\begin{remark}
\label{liftsofautomorphisms} 
It is a simple calculation that the lifts of the $\bar\alpha_\sigma$ and $\bar\alpha_i^j$ satisfy all relations of the presentation of $\Aut(W)$ given in Section~\ref{section:presentation-of-autw} except those of the second type, i.e. the relations $(\bar\alpha_i^j)^2=\hbox{id}$. Indeed $(\alpha_i^j)^2(x_i)=x_j^2x_ix_j^{-2}=x_ix_j^{-4}$ and $(\alpha_i^j)^2(x_k)=x_k$ for $k\neq i$ which implies that $(\alpha_i^j)^2=t_{\mathfrak a}$ with $\mathfrak a=-2\varepsilon_{ij}$.
\end{remark}

We conclude this section with a brief discussion of the outer automorphism group of $G$. We consider $A$ as a subgroup of $M^*$ via the inclusion 
\begin{equation}
\label{eq:iota}
\iota:A\to M^*,\ a = (x_1^2)^{a_1}\cdots (x_n^2)^{a_n}\mapsto [a_{ij}],
\end{equation} 
where
\[
a_{ij} = \left\{
\begin{array}{rll}
0 && \text{if } i=j,\\
-2a_i && \text{otherwise.}
\end{array}
\right.
\]

It follows easily that the translation automorphism $t_{\iota(a)}$ is simply conjugation by $a$ and therefore an inner automorphism of $G$. For all $\mathfrak a\in M^*\setminus \iota(A)$ some $t_{\mathfrak a}(x_i)$ is conjugate to a proper power of $x_i$ which implies that $t_{\mathfrak a}$ is not inner. Thus $t(M^*)\cap \Inn(G)=t(\iota(A))$. In particular we have the monomorphism 
\[
M^*/\iota(A)\to \Out(G),\ \mathfrak a \iota(A)\mapsto t_{\mathfrak a}\Inn(G).
\]
\begin{fdiagram}[ht]
\[ 
\begin{tikzcd}
&1\arrow{d}&1\arrow{d}&1\arrow{d}&\\
1\arrow{r}&A \arrow{r}\arrow{d}{\iota}&G\arrow{r}\arrow{d}&W\arrow{r}\arrow{d}&1\\
1\arrow{r}&M^* \arrow{r}{t}\arrow{d}&\Aut(G)\arrow{r}\arrow{d}&\Aut(W)\arrow{r}\arrow{d}&1\\
1\arrow{r}&M^*/\iota(A) \arrow{r}\arrow{d}&\Out(G)\arrow{r}\arrow{d}&\Out(W)\arrow{r}\arrow{d}&1\\
&1&1&1&
\end{tikzcd}
\]
\caption{Automorphisms and outer automorphisms of $G$}
\label{diag:basic-1}
\end{fdiagram}

\begin{theorem}
\label{theorem:aut-out-1}
Assume $n \geq 3$. Then Diagram \ref{diag:basic-1} is commutative and every row and column is exact.
\end{theorem}

\begin{proof}
Since by \cite[Proposition 5]{PS22} the center of $G$ is trivial, we have the monomorphism $A \to \Inn(G)$ and the isomorphism $G \cong \Inn(G)$. As the free product of cyclic groups has trivial center it follows moreover that $W \cong \Inn(W)$. We therefore get the commutativity and exactness in the first two rows of the diagram. The rest is a simple diagram chase and follows from the isomorphism theorem and the discussion preceding this Lemma.
\end{proof}

\section{Presentation of \texorpdfstring{$\Aut(G)$}{Aut(G)}}

In this section we collect the facts established in the previous one to give finite presentations of $\Aut(G)$ and $\Out(G)$. The generating set of $\Aut(G)$ consists of the generators of $M^*$ discussed in Section \ref{section:translation-endomorphisms} and the lifts of the generators of $\Aut(W)$ introduced in the proof of Theorem~\ref{theorem:aut-w-1}. 
Thus the generating set consists of the automorphisms 
\begin{equation}
\label{equation:generators-of-autg}
\alpha_\sigma, \alpha_i^j, \varepsilon_{ij}, \delta_i,
\end{equation}
where $\sigma \in S_n$, $1 \leq i,j \leq n$ and $i \neq j$.
\begin{remark}
From now on, whenever we use the automorphisms introduced above, we assume that $1 \leq i,j \leq n$ and $i \neq j$, even without stating it explicitly.
\end{remark}

The previous sections imply that the automorphisms are defined in the following way. We give only their images on those generators of $G$, which they do not fix:
\begin{itemize}
\item $\alpha_{\sigma}(x_i) = x_{\sigma(i)}$
\item $\alpha_i^j(x_i) = x_j x_i x_j^{-1}$
\item $\varepsilon_{ij}(x_i) = x_i x_j^2$
\item $\delta_i(x_i) = x_i^{-1}$
\end{itemize}
 
The following is an immediate consequence of Proposition~\ref{proposition:shift-automorphisms}, the presentation of $\Aut(W)$ and Remark~\ref{liftsofautomorphisms} and some direct calculation.

\begin{theorem}
\label{theorem:presentation-of-autg}
The automorphism group of $G$ is generated by automorphisms from \eqref{equation:generators-of-autg}. A presentation of $\Aut(G)$ is then given by the following three sets of relations:
\begin{enumerate}[label=\arabic*.]
\item Relations coming from $M^*$:
\[
[\varepsilon_{ij},\varepsilon_{kl}]=1, \quad [\delta_i,\delta_j]=1, \quad \delta_i^2=1, \quad
\delta_i \varepsilon_{kl} \delta_i^{-1} = \left\{
\begin{array}{ll}
\varepsilon_{kl} & \text{if } l \neq i\\
\varepsilon_{kl}^{-1} & \text{if } l = i\\
\end{array}
\right.
\]
for all $1 \leq i,j,k,l \leq n$.

\item Relations coming from $\Aut(W)$:
\[
\alpha_\sigma \alpha_\tau = \alpha_{\sigma\tau}, \quad 
(\alpha_i^j)^2 = \varepsilon_{ij}^{-2}, \quad
[\alpha_i^j, \alpha_k^l] = 1, \quad
[\alpha_i^j\alpha_m^j, \alpha_i^m] = 1, \quad
\alpha_\sigma \alpha_i^j \alpha_\sigma^{-1} = \alpha_{\sigma(i)}^{\sigma(j)}
\]
for all $\sigma,\tau \in S_n$ and $1 \leq i,j,k,l,m \leq n$, with additional assumption that $j \neq k$ and $i \neq l$.

\item Relations coming from the action of lifts of generators of $\Aut(W)$ on $M^*$:
\begin{align*}
\alpha_\sigma \delta_i &= \delta_{\sigma(i)}\alpha_\sigma, &
[\alpha_i^j, \delta_{k}] &= 1, &
\alpha_i^j \delta_{j} &= \delta_j \varepsilon_{il}^{2} \alpha_i^j
\\
\alpha_\sigma \varepsilon_{kl} &= \varepsilon_{\sigma(k) \sigma(l)} \alpha_\sigma, &
[\alpha_i^j, \varepsilon_{kl}] &= 1, & 
\alpha_i^j \varepsilon_{jl} &= \varepsilon_{jl} \varepsilon_{il}^2 \alpha_i^j,
\\
&& &&
\alpha_i^j \varepsilon_{mi} &= \varepsilon_{mi}^{-1} \alpha_i^j,
\end{align*}
for all $\sigma \in S_n, 1 \leq i,j,k,l,m \leq n$, with additional assumption that $j \neq k$ and $i \neq l$.
\end{enumerate}
\end{theorem}

By the description of the inclusion $\iota:A \to M^*$ preceding Theorem \ref{theorem:aut-out-1} and by Section \ref{section:presentation-of-autw}, we can write down the presentation of $\Out(G)$:

\begin{theorem}
\label{theorem:presentation-of-outg}
The outer automorphism group of $G$ is generated by the automorphisms from \eqref{equation:generators-of-autg}, which satisfy relations given in Theorem \ref{theorem:presentation-of-autg}, with the addition of the following two groups:
\begin{itemize}
    \item[1.\!\!'] Relations coming from the inclusion $A \stackrel{\iota}{\to} M^* \stackrel{t}{\to} \Aut(G)$:
    \[
    \varepsilon_{i1}^2 \cdots \varepsilon_{i,i-1}^2 \varepsilon_{i,i+1}^2 \cdots \varepsilon_{in}^2,
    \]
    for all $1 \leq i \leq n$.
    \item[2.\!\!'] Relations coming from the inclusion $W \to \Inn(W) \to \Aut(W)$:
    \[
    \alpha_1^j \cdots \alpha_{j-1}^j \alpha_{j+1}^j \cdots \alpha_n^j,
    \]
    for all $1 \leq j \leq n$.
\end{itemize}
\end{theorem}

\section{Semi-linear automorphisms of \texorpdfstring{$A$}{A}}

We start with collecting some elementary, but important facts on the structure of $W$-module $A$, defined by \eqref{equation:action-of-w}. Using the $\Z$-basis $(x_1^2,\ldots,x_n^2)$ of $A$ we get, that the action of $W$ on $A$ corresponds to the representation $\rho \colon W \to \GL_n(\Z)$ defined by
\[
\rho(\bar x_i) = \diag(\underbrace{-1,\ldots,-1}_{i-1},1,-1,\ldots,-1)
\]
for every $1 \leq i \leq n$. 
It follows that $A$ decomposes into a direct sum of $n$ copies of non-isomorphic modules of rank $1$. Since none of these submodules is trivial, the action of $W$ on $A$ is fixed-point free:
\begin{corollary}
\label{corollary:a^w}
$A^W := \{ a \in A : \forall_{\bar w \in W} \; \bar w \cdot a = a  \} = 1$.
\end{corollary}
\noindent
Moreover, by Schur's lemma \cite[Lemma I.8.1]{Feit82}, the structure of the endomorphism ring of $A$ is given by 
\[
\End_W(A) \cong \Z^n
\]
and with the given basis, it is just a set of diagonal integer matrices. We immediately get
\begin{corollary}
\label{corollary:W-automorhisms}
$\Aut_W(A) = \{ \diag(\pm 1, \ldots, \pm 1) \} \cong C_2^n$.
\end{corollary}

\begin{remark}
Note that for a natural number $k$, $C_k$ and $\Z/k$ both denote cyclic groups of order $k$. We use different notation to distinguish between the multiplicative ($C_k$) and the additive ($\Z/k$) notation.
\end{remark}

We will further exploit the diagonal action of $W$ in the description of the second cohomology group of $A$. By \cite[Corollary 6.2.10]{We94} we have that
\begin{equation}
\label{equation:decomposition-of-h2}
H^2(W,A) \cong \bigoplus_{i=1}^n H^2(C_2,A).
\end{equation}
The action of each $C_2$ on $A$ has exactly one trivial summand, hence
\begin{equation}
\label{equation:cohomology-of-c2}
H^2(C_2,A) \cong \Z/2.
\end{equation}

Every class $\alpha \in H^2(W,A)$ defines some extension of $A$ by $W$. Because $A$ is a free abelian group, this extension will have torsion if and only if the restriction of $\alpha$ to some subgroup of $W$ of finite order will be zero. It is easy to check that it is enough to focus on representatives of conjugacy classes of such groups. Hence, the decomposition \eqref{equation:decomposition-of-h2} and the formula \eqref{equation:cohomology-of-c2} give us:

\begin{corollary}
\label{corollary:cohomology-class}
There is only one element in $H^2(W, A)$ which defines a torsion-free extension.
\end{corollary}

To compute the groups of automorphisms and outer automorphisms in some specific cases, the concept of semi-linear automorphism has been introduced in \cite[Definition III.2.2]{Ch86}. For convenience, we adapt it to our case:

\begin{definition}
A \emph{semi-linear} automorphism of $W$-module $A$ is a pair $(f,F) \in \Aut(A) \times \Aut(W)$, such that for every $a \in A$ and $w \in W$ we have
\[
f(w \cdot a) = F(w) \cdot f(a).
\] 
The group of semi-linear automorphisms of $A$ (with component-wise multiplication) is denoted by $\Aut_S(A)$.
\end{definition}

Recall that $\FR(W)$ is the Fouxe-Rabinovitch subgroup of $\Aut(W)$ (see Section \ref{section:presentation-of-autw}).
Since the action of $W$ on $A$ factors through an abelian group, we get the following lemma:
\begin{lemma}
\label{lemma:twisted-action-of-w}
Let $F \in \FR(W)$. Then for every $w \in W$ and $a \in A$ we have
\[
F(w) \cdot a = w \cdot a.
\]
\end{lemma}

\begin{proposition}
\label{proposition:structure-of-aut_s-w}
We have the following split extension of groups:
\[
1 \longrightarrow \Aut_W(A) \times \FR(W) \longrightarrow \Aut_S(W) \stackrel{\pi}{\longrightarrow} S_n \longrightarrow 1.
\]
\end{proposition}
\begin{proof}
The epimorphism onto $S_n$ is given by the composition of the projection to the second factor and the epimorphism in \eqref{equation:aut_w}. The obvious section is given by the formula $\sigma \mapsto ({\alpha_\sigma}_{|A}, \bar\alpha_\sigma)$. 

The inclusion $\Aut_W(A) \times \FR(W) \subset \ker \pi$ is straightforward. Let $(f,F) \in \ker \pi$. By \eqref{equation:aut_w}, $F \in \FR(W)$, hence by Lemma \ref{lemma:twisted-action-of-w}
\[
f (w \cdot a) = F(w) \cdot f(a) = w \cdot f(a)
\]
for every $w \in W, a \in A$ and in the consequence, $f \in \Aut_W(A)$.
\end{proof}

Proposition \ref{proposition:structure-of-aut_s-w} gives us some insight into the structure of $\Aut_S(A)$.
\begin{corollary}
\label{corollary:splits}
Short exact sequences
\[
1 \longrightarrow \Aut_W(A) \longrightarrow \Aut_S(A) \stackrel{\pi}{\longrightarrow} \Aut(W) \longrightarrow 1
\]
and
\[
1 \longrightarrow \Aut_W(A) \longrightarrow \Aut_S(A)/W \stackrel{\pi}{\longrightarrow} \Out(W) \longrightarrow 1
\]
both split.
\end{corollary}
\begin{proof}
It is enough to note that $\FR(W)$ is a normal subgroup of $\Aut_S(A)$ and then use the short exact sequence \eqref{equation:aut_w}. The splitting of the second sequence is obvious when one notes the inclusion
\[
W \cong \Inn(W) \subset \FR(W).
\]
\end{proof}

\begin{remark}
\label{remark:action-of-semilinear-automorphisms}
The group $\Aut_S(A)$ acts on $H^2(W,A)$ as follows. If $\gamma \in H^2(W,A)$, $c$ is a representative $2$-cocycle of $\gamma$ and $\eta = (f,F) \in \Aut_S(A)$, then
\[
\eta * \gamma = [ \eta * c ],
\]
where
\[
\eta * c(w_1,w_2) = fc(F^{-1}(w_1),F^{-1}(w_2))
\]
for every $w_1,w_2 \in W$ (see \cite[page 201]{Ch86}).
\end{remark}

Assume that $\gamma \in H^2(W, A)$ is the cohomology class defining the extension \eqref{eq:extension}. Denote by $\Aut_S(A)_\gamma$ the stabilizer of $\gamma$ in $\Aut_S(A)$ under the action defined in Remark \ref{remark:action-of-semilinear-automorphisms}.

\begin{corollary}
$\Aut_S(A)_\gamma = \Aut_S(A)$.
\end{corollary}
\begin{proof}
Let $\eta \in \Aut_S(A)$. By assumption, $\gamma$ and hence $\eta * \gamma$ both define torsion-free extensions of $A$ by $W$. By Corollary \ref{corollary:cohomology-class}, only one element of $H^2(W, A)$ has this property and hence $\eta * \gamma = \gamma$.
\end{proof}

\section{Automorphisms of \texorpdfstring{$G$}{G}. Second approach}

By Corollary \ref{corollary:a^w} and \cite[Proposition 5]{PS22}, $A^W = Z(G)$, since both of these groups are trivial. Using Corollary \ref{corollary:characteristic} allows us to use the top diagram from page 201 of \cite{Ch86}, which in particular gives us another view of $\Aut(G)$ and $\Out(G)$ than the one given by Theorem \ref{theorem:aut-out-1}. Recall that $\Aut^0(G)$ is the group of those automorphisms of $G$ which induce identities on both $A$ and $W$. Referring to the notation introduced in Section \ref{section:translation-endomorphisms}, we begin by making the following observation:
\begin{lemma}
\label{lemma:aut-0}
$\Aut^0(G) = t(M_0)$.
\end{lemma}
\begin{proof}
By Theorem \ref{theorem:aut-w-1}, $\Aut^0(G) \subset \ker \pi = t(M^*)$. Let $t_a = t([a_{ij}])$ for some $[a_{ij}] \in M^*$. By formula \eqref{equation:square-of-shift}, we have
\[
t_a(x_i^2) = (x_i^2)^{2a_{ii}+1}.
\]
We demand from $t_a$ to induce the identity on $A$, hence $(x_i^2)^{2a_{ii}}=1$. Since this is an element of a free abelian group, we get $a_{ii}=0$, as desired.
\end{proof}

\begin{fdiagram}[ht]
\[ 
\begin{tikzcd}
&1\arrow{d}&1\arrow{d}&1\arrow{d}&\\
1\arrow{r}&A \arrow{r}\arrow{d}{\iota}&G\arrow{r}\arrow{d}&W\arrow{r}\arrow{d}&1\\
1\arrow{r}&M_0 \arrow{r}{t}\arrow{d}&\Aut(G)\arrow{r}\arrow{d}&\Aut_S(A)\arrow{r}\arrow{d}&1\\
1\arrow{r}&H^1(W,A) \arrow{r}\arrow{d}&\Out(G)\arrow{r}\arrow{d}& \Aut_W(A) \rtimes \Out(W)\arrow{r}\arrow{d}&1\\
&1&1&1&
\end{tikzcd}
\]
\caption{Automorphisms and outer automorphisms of $G$}
\label{diag:basic-2}
\end{fdiagram}

\begin{theorem}
\label{theorem:aut-out-2}
Assume $n \geq 3$. Diagram \ref{diag:basic-2} is commutative, where every row and column is exact. In particular, up to isomorphism, the bottom row looks as follows:
\[
1 \longrightarrow \Z^{n(n-2)} \times (\Z/2)^n \longrightarrow \Out(G) \longrightarrow (\Z/2)^n \rtimes \Out(W) \longrightarrow 1.
\]
\end{theorem}
\begin{proof}
Diagram \ref{diag:basic-2} is exactly the before mentioned diagram from the top of page 201 of \cite{Ch86}. The isomorphism type of $H^1(W, A)$ is given by Lemma \ref{lemma:aut-0}, Proposition \ref{proposition:shift-automorphisms}, inclusion $\iota \colon A \to M^*$ defined in \eqref{eq:iota} and the elementary divisor theorem. The right-hand side of the short exact sequence follows from Corollary \ref{corollary:splits}.
\end{proof}

\section*{Acknowledgements}

The main part of the work was done at the Kiel University. The first two authors thank the Department of Mathematics of the university for the hospitality. The work was supported by the SEA-EU Alliance.

\bibliographystyle{alpha}
\bibliography{bibl}

\end{document}